\renewcommand{\P}{\mathbb{P}}
\newcommand{\E}{\mathbb{E}}
\documentclass[12pt]{article}
\usepackage{amsmath,amssymb}
\usepackage{graphicx}
\usepackage[dvips]{epsfig} 
\usepackage{amsmath,bm}
\usepackage[center]{subfigure}
\usepackage{amsfonts,amsthm}

\usepackage{times}

 \usepackage{mathptmx}
\usepackage{latexsym}
\usepackage[latin1]{inputenc}
\usepackage[T1]{fontenc}
\usepackage{amsfonts}
\usepackage{amsthm}
\usepackage[mathscr]{eucal}
\usepackage[english,francais]{babel}
\usepackage{verbatim}

\usepackage{appendix}
\usepackage{color}

 \newcommand{\ds}{\displaystyle}

\newtheorem{thm}{Theorem}[section]
\newtheorem{lea} [thm]{Lemma} 

\newtheorem{cor}[thm]{Corollary} 

\numberwithin{equation}{section} 
\numberwithin{thm}{section}
\numberwithin{rem}{section}
\numberwithin{ex}{section}
\begin{document}
\title{On convergence of the sample correlation matrices in high-dimensional data}
\date{}
\author{S\'ev\'erien Nkurunziza{\thanks {University of Windsor, 401 Sunset
Avenue, Windsor, Ontario, N9B 3P4. Email: severien@uwindsor.ca}} \quad{ } and \quad{ } Yueleng Wang{\thanks {University of Windsor, 401 Sunset
Avenue, Windsor, Ontario, N9B 3P4. Email: wang1gv@uwindsor.ca}}}
\thispagestyle{empty} \selectlanguage{english} \maketitle
\maketitle\thispagestyle{empty}

\begin{abstract}
In this paper, we consider an estimation problem concerning the matrix of correlation coefficients in context of high dimensional data settings.
In particular, we revisit some results in Li and Rolsalsky [Li, D. and Rolsalsky, A. (2006). {\it Some strong limit theorems for the largest entries
of sample correlation matrices},  The Annals of Applied Probability, {\bf 16},  1, 423--447]. Four of the main theorems of Li and Rolsalsky~(2006) are established in their full generalities and we simplify substantially some proofs of the quoted paper. Further, we generalize a theorem which is useful in deriving the existence
of the $p^{\mbox{th}}$ moment as well as in studying the convergence rates in law of large numbers.
 \end{abstract}





 \pagenumbering{arabic} \addtocounter{page}{0}
\section{Introduction}
As in Li and Rolsalsky (2006), we consider an estimation problem concerning the matrix of correlation coefficients in context of high-dimensional data. In particular, as in the quoted paper, we are interested in asymptotic properties of the largest entries of the matrix of the correlation coefficients when the sample size may be smaller than the parameters.
Thus, we use the same notations as in Li and Rolsalsky (2006). Namely, let $X= (X_1,X_2,...,X_p)$ be a $p$-variate random vector, and $\textbf{M}_{n,p_n} = (X_{k,i})_{1\leq k\leq n,1\leq i\leq p_n}$, rows of $\textbf{M}_{n,p_n}$ are independent copies of $\textbf{X}$. Let $\mathcal{M} = \{X_{k,i}; i\geq 1,k\geq 1\}$ be an array of i.i.d. random variables. Further, let $\bar{X}_{i}^{(n)} = \ds{\sum_{k=1}^n}X_{k,i}/n,$ $ \textbf{X}_i^{(n)}$ means $i$th column of $\textbf{M}_{n,p_n}$, and let $W_n = \ds{\max_{1\leq i< j \leq p_n}}\big|\sum_{k=1}^nX_{k,i}X_{k,j}\big|$. Li and Rosalsky~(2006) studied the limit behavior of $W_n$ and that of $L_n$ which is defined as
\begin{align*}
 L_n= \max_{1\leq i < j\leq p_n}|\hat{\rho}_{i,j}^{(n)}| \quad{} \mbox{with } \quad{ }
  \hat{\rho}_{i,j}^{(n)}= \ds{\frac{\ds{\sum_{k=1}^{n}}(X_{k,i} - \bar{X}_i^{(n)})(X_{k,j} - \bar{X}_j^{(n)})}{\big(\ds{\sum_{k=1}^{n}}(X_{k,i}- \bar{X}_i^{(n)})^2\big)^{1/2}\big(\ds{\sum_{k=1}^{n}}(X_{k,j}- \bar{X}_j^{(n)})^2\big)^{1/2}}}.
 \end{align*}
In Jiang~(2004), the author derived the asymptotic properties of the statistic $L_n$ and he proposed a test for testing if the components of the p-column vector $X$ are uncorrelated.

In this paper, we establish the asymptotic results which refine the analysis in Li and Rosalsky~(2006). More specifically, with respect to the similar work in literature, we extend the existing findings in three ways. First, Theorems 2.1 and 2.3 of Li and Rosalsky (2006) are established in their full generalities. Thanks to the established results, we also revisit the statement given in Remark 2.1 in Li and Rosalsky (2006). Further, from the established results, we simplify remarkably the proof of the result stated in Remark 2.3 of Li and Rosalsky (2006).
More precisely, we prove that the condition which is stated as sufficient is also necessary. Second, we refine Theorems~3.2 and Theorem~3.3 of Li and Rosalsky (2006) and, as compared to the proofs given in Li and Rosalsky (2006), we provide the proofs which are significantly shorter than that in the quoted paper.  Third, we generalize Theorem 3.2.1 in Chung (1974) which is useful in establishing the existence of the $p^{\mbox{th}}$-moment as well as in studying the convergence rates in Law of Large numbers. Specifically, the established result is useful in deriving the main results in Erdös (1949), Baum and Katz (1965) and Katz (1963) among others.

The remainder of this paper is organized as follows. Section~\ref{sec:prelim} presents some preliminary results.
In Section~\ref{sec:mainres}, we present the main results of this paper. 
Finally, for the convenience of the reader, we present some proofs in the Appendix~\ref{sec:appendA} and, we recall in the Appendix~\ref{sec:appendB}
some existing results used in this paper.
\section{Some Preliminary Results}\label{sec:prelim}
In this section, we derive some results which are useful in establishing the main results of this paper. In particular, the results of this section generalize Theorem 3.2.1 given in Chung~(1974). For the convenience of the reader, the quoted theorem stipulates that for a random variable $X$, we have $\ds{\sum_{n=1}^\infty} \P(|X|\geqslant n)\leqslant\E(|X|)\leqslant 1+\ds{\sum_{n=1}^\infty} \P(|X|\geqslant n)$. Thus,\\ $\E(|X|)<\infty$ if and only if $\ds{\sum_{n=1}^\infty} \P(|X|\geqslant n)<\infty$. To introduce some notations, let $a_{n}=\mathcal{O}(b_n)$  stand for the sequence $a_{n}/b_{n}$, n=1,2,\dots is bounded.

\begin{thm}
  \label{thm:2}
  Let $\{\alpha_n\}_{n=0}^{\infty}$ be nonnegative sequence of real numbers, let $\{\beta_n\}_{n=0}^{\infty}$ be nonnegative and nondecreasing sequence of real numbers such that 
  $\beta_{0}=0$, $\ds{\lim_{n\to\infty}}\beta_n = \infty,$ \\$\beta_{n+1}-\beta_n = \mathcal{O}(\beta_n)$ and
   $ c^{-1}\alpha_n\leq \beta_n - \beta_{n-1} \leq c\alpha_n \quad n=1,2,\dots, \text{  for some }\\ c\geqslant1.$
  Then, for any random variable $X$,\\
    $c^{-1}\ds{\sum_{n=1}^\infty} \alpha_n\P(|X|\geqslant \beta_n)\leqslant \E(|X|)\leqslant \beta_{1}+(B+1)c\ds{\sum_{n=1}^\infty} \alpha_n\P(|X|\geqslant \beta_n),$ for some $B > 0$.\\
   Thus,
    $\E(|X|) < \infty$
   if and only if
    $\ds{\sum_{n=1}^\infty} \alpha_n\P(|X|\geqslant \beta_n) <\infty.$
\end{thm}
The proof of this theorem is given in the Appendix \ref{sec:appendA}. To illustrate how the established result generalizes Theorem~3.2.1 in Chung~(1974), we note first that, from the quoted theorem, we also have, $\ds{\sum_{n=1}^\infty} \P(|X|\geqslant n)\leqslant\E(|X|)\leqslant 1+2\ds{\sum_{n=1}^\infty} \P(|X|\geqslant n)$.
For the particular case where $\alpha_{n}=1$, $\beta_{n}=n$, Theorem~\ref{thm:2} yields Theorem~3.2.1 in Chung~(1974) with $c=B=1$. 
By using this theorem, we establish the following
result which is useful in deriving one of the main result of this paper.
\begin{cor} \label{coro:1}
 Let $\alpha>0$, let $\beta>0$ and let $X$ be a random variable. Then, \\ $\ds{\sum_{n=1}^\infty} n^\alpha\P\Big(|X|>n^\beta\Big) < \infty$ if and only if $\E\Big( |X|^{\frac{\alpha+1}{\beta}} \Big) < \infty$.
\end{cor}
The proof of this corollary is given in the Appendix \ref{sec:appendA}.  Further, from Theorem~\ref{thm:2} and Corollary~\ref{coro:1}, we derive the following corollary which is useful in establishing the second main result of this paper.

\begin{cor} \label{coro:2}
  Let $\alpha,\beta>0$, let  $\{\alpha_n\}_{n=1}^{\infty}$ and $\{\beta_n\}_{n=1}^{\infty}$ be nonnegative sequences of real numbers and suppose that $\alpha_n/n^\alpha$ and $\beta_n/n^\beta$ are bounded away from $0$ and $\infty$. Then, for a random variable $X$, $\ds{\sum_{n=1}^\infty} \alpha_n\P\Big( |X|>\beta_n \Big) < \infty$ if and only if $\E\Big( |X|^{\frac{\alpha+1}{\beta}} \Big) < \infty$.
\end{cor}
The proof follows directly from Corollary \ref{coro:1}. Below, we derive a lemma which plays a central role in establishing the main results. Thank to the established lemma, we also simplify significantly the proof of the statement in Remark~2.3 of Li and Rosalsky~(2006). To introduce some notations, we consider that  $X_{1,i},i=1,2,...,$ are independent  and identically distributed random variables, and $\textbf{X}_k = (X_{k,1},X_{k,2},...,)$ is an independent copy of a random vector $\textbf{X}_{1}= \left(X_{1,1},X_{1,2},...,\right)$.
\begin{lea} \label{lem:1}
   Let $m$ be a fixed positive integer and let $\left\{u_{n}\right\}_{n=1}^{\infty}$ be a nonnegative 
   sequence of real numbers. 
   Then, \\$\ds{\sum_{n=1}^\infty} n^m\P\left(\prod_{h=1}^{m}|X_{1,h}|\geq u_{n}\right)< \infty$ if and only if $\ds{\sum_{n=1}^\infty}\P\left(\max_{1\leq i_{1}<i_{2}<\dots<i_{m}\leq n}\prod_{h=1}^{m}|X_{1,i_{h}}|\geq u_n\right) < \infty$.
\end{lea}
The proof of this lemma is given in the Appendix \ref{sec:appendA}. Note that the if part is obvious as it follows directly from the sub-additivity of a probability measure.
We also prove the following lemma which is useful in deriving the main results of this paper.
\begin{lea} \label{lem:11}
   Let $m$ be a fixed positive integer. Let $\left\{u_{n}\right\}_{n=1}^{\infty}$ be a nonnegative and nondecreasing sequence of real numbers such that
   $\ds{\lim_{n\rightarrow\infty}}u_{n}=\infty$. Further, suppose that there exists a nonnegative, continuous and increasing function $f$ such that $f(u_{n})\Big/n^{\beta}$ is bounded away from 0 and from infinity, for some $\beta>0$. Then, \\$\ds{\sum_{n=1}^\infty}\P\left(\max_{1\leq i_{1}<i_{2}<\dots<i_{m}\leq n}\prod_{h=1}^{m}|X_{1,i_{h}}|\geq u_n\right) < \infty$ if and only if $\E\left(\left[f\left(\ds{\prod_{h=1}^{m}}|X_{1,h}|\right)\right]^{\frac{m+1}{\beta}}\right)< \infty$.
\end{lea}
\begin{proof}[Proof]
From Lemma~\ref{lem:1}, $\ds{\sum_{n=1}^\infty}\P\left(\max_{1\leq i_{1}<i_{2}<\dots<i_{m}\leq n}\prod_{h=1}^{m}|X_{1,i_{h}}|\geq u_n\right) < \infty$ if and only if \\$\ds{\sum_{n=1}^\infty} n^m\P\left(\prod_{h=1}^{m}|X_{1,h}|\geq u_{n}\right)< \infty$. Since the function $f$ is increasing and continuous, this last statement is equivalent to $\ds{\sum_{n=1}^\infty} n^m\P\left[f\left(\prod_{h=1}^{m}|X_{1,h}|\right)\geq f(u_{n})\right]< \infty$. Then, by using Corollary~\ref{coro:2},
this last statement is equivalent  to $\E\left[\left(f\left(\ds{\prod_{h=1}^{m}}|X_{1,h}|\right)\right)^{\frac{m+1}{\beta}}\right]<\infty$, this completes the proof.
\end{proof}
\begin{cor} \label{cormainl:2}
    We have  $\ds{\sum_{n=1}^\infty}\P\left(\max_{1\leq i_{1}<i_{2}<\dots<i_{m}\leq n}\prod_{h=1}^{m}|X_{1,i_{h}}|\geq \sqrt{n\ln(n)}\right) < \infty$ if and only if $\E\left[\ds{\prod_{h=1}^{m}}|X_{1,h}|^{2(m+1)}\Bigg /\left(\ln\left(e+\ds{\prod_{h=1}^{m}}|X_{1,h}|\right)\right)^{m+1}\right]<\infty$, with $m$ a fixed positive integer.
\end{cor}
The proof of this corollary is given in the Appendix \ref{sec:appendA}. From this corollary, we establish the following result which improves the statement in Remark 2.3 of Li and Rosalsky~(2006).
For the convenience of the reader, we recall that, in Remark 2.3 in Li and Rosalsky~(2006), the authors conclude that $\ds{\sum_{n=1}^\infty}\P\Big( \max_{1\leq i<j\leq n}|X_1X_2|\geq \sqrt{n\log n} \Big)<\infty$ implies $\E|X_1|^\beta<\infty,$ for $0\leqslant\beta<6$. This becomes a special case of the following result by taking $m=2$. 

\begin{cor} \label{correm23}
    Suppose that $\ds{\sum_{n=1}^\infty}\P\left(\max_{1\leq i_{1}<i_{2}<\dots<i_{m}\leq n}\prod_{h=1}^{m}|X_{1,i_{h}}|\geq \sqrt{n\ln(n)}\right) < \infty$, for a fixed positive integer $m$. Then, $\E\left[|X_{1,1}|^{\beta}\right]<\infty$, for all $0\leqslant \beta<2(m+1)$.
\end{cor}
The proof of this corollary follows directly from Corollary~\ref{cormainl:2} along with classical properties of expected value of random variables. 
\begin{cor} \label{cormainlog}
   Suppose that the conditions of Lemma~\ref{lem:11} hold with $u_{n}/n^{\beta}$ bounded away from 0 and from infinity, for some $\beta>0$. Then,  $\ds{\sum_{n=1}^\infty}\P\left(\max_{1\leq i_{1}<i_{2}<\dots<i_{m}\leq n}\prod_{h=1}^{m}|X_{1,i_{h}}|\geq u_n\right) < \infty$ if and only if $\E\left(|X_{1,1}|^{\frac{m+1}{\beta}}\right)<\infty$.
\end{cor}
\begin{proof}
  This result follows immediately from Lemma~\ref{lem:11} with $f$ being an identity function, and the fact that $X_{1,i},i=1,2,...$ are  i.i.d. random variables.
\end{proof}

\begin{cor} \label{cormain:1}
   Let $\left\{u_{n}\right\}_{n=1}^{\infty}$ be a nonnegative sequence of real numbers. 
   Then, \\$\ds{\sum_{n=1}^\infty} n^2\P\left(|X_{1,1}X_{1,2}|\geq u_{n}\right)< \infty$ if and only if $\ds{\sum_{n=1}^\infty}\P\left(\max_{1\leq i<j\leq n}|X_{1,i}X_{1,j}|\geq u_n\right) < \infty$.
\end{cor}
\begin{proof}[Proof]
The proof follows directly from Lemma~\ref{lem:1} by taking $m=2$.
  \end{proof}
\section{Main Results}\label{sec:mainres}
In this section, we present the main results of this paper. In particular,
Theorems~2.1~and~2.3 of Li and Rosalsky~(2006) are established
in their full generalities. We also refine Theorem~3.2 and 3.3 in Li and Rosalsky~(2006) and we provide the proofs which are significantly shorter than that given in the quoted paper. In the sequel, as in Li and Rosalsky~(2006),  let $\left\{(U_{k,i},V_{k,i}); i\geqslant 1, k\geqslant 1\right\}$ be iid two-dimensional random vectors, and let $\{p_{n}, n\geqslant 1\}$ be a sequence of positive integers. Further,let $T_{n}=\ds{\max_{1\leqslant i\neq j\leqslant p_{n}}}\left|\sum_{k=1}^{n}U_{k,i}V_{k,j}\right|$, $n=1,2,\dots$, let $\{Y_{n},n=1,2,\dots\}$ be a sequence of iid random variables where $Y_{1}$ is distributed as $U_{1,1}V_{1,2}$ and let $S_{n}=\ds{\sum_{k=1}^{n}}Y_{k}$, $n=1,2,\dots$

\begin{thm}\label{thm:li3.2}
  Suppose that $n/p_n$ is bounded away from $0$ and $\infty$, and let $1/2<\alpha\leq1$. Then, the following conditions are equivalent.
  \begin{description}
    \item[(1.)] $\ds{\lim_{n\to\infty}}\frac{T_n}{n^\alpha} = 0$ a.s.
    \item[(2).] $\ds{\sum_{n=1}^\infty }\P\left( \max_{1\leq i\neq j\leq n}|U_{1,i}V_{1,j}|\geq n^\alpha \right)<\infty \quad \text{and}\quad \E(U_{1,1})\E V_{1,1})=0$.
    \item[(3).] $\E\left(|U_{1,1}|^{3/\alpha}|V_{1,2}|^{3/\alpha}\right)<\infty$ and $\E(U_{1,1})\E V_{1,1})=0$.
  \end{description}
\end{thm}
\begin{proof}
  The equivalence between parts (2) and (3) follows directly from Corollary~\ref{cormainlog} by taking $m=2$ and $X_{1,1}=U_{1,1}V_{1,2}$. As for the equivalence between part (1) and part (2), the proof of "only if" part is similar to that given in Li and Rosalsky (2006) and thus, we need to give the proof of the "if" part. To this end, note that by Corollary~\ref{cormain:1} and Corollary~\ref{coro:2}, the condition $\displaystyle{\sum_{n=1}^\infty \P\left( \max_{1\leq i\neq j\leq n}|U_{1,i}V_{1,j}|\geq n^\alpha \right)<\infty}$ is equivalent to $\E|U_{1,1}V_{1,2}|^\frac{3}{\alpha}<\infty$. Then, by the celebrated theorem of Baum and Katz~(1965) (or see Theorem~\ref{thm:BaumKatz} in the Appendix~\ref{sec:appendB}), we have, 
    $\ds{\sum_{n=1}^\infty }n\P\left(\ds{\sup_{m\geqslant n}}\frac{|S_m|}{m^\alpha}>\epsilon\right) < \infty \quad \text{for all }\epsilon >0$.
  Then, since $c^{-1}\leq p_n/n<c$, $n\geq1$, which means $\frac{p_n^2}{n}< \frac{c^2n^2}{n}=c^2n$, then
  \begin{align}\label{thm:li3.2eq1}
    \sum_{n=1}^\infty \frac{p_n^2}{n}\P\left(\frac{|S_n|}{n^\alpha}>\epsilon\right) < \infty, \quad \text{for all }\epsilon >0.
  \end{align}
  Further, one verifies that $\displaystyle{\lim_{c\downarrow1}\limsup_{n\to\infty}\frac{[cn]^\alpha}{n^\alpha}=1}$. Futher, note that $p_n^2/n>c^{-2}n$, and then, by \eqref{thm:li3.2eq1}, we have  $c^{-2}n\P\left(\frac{|S_n|}{n^\alpha}>\epsilon\right)\to 0$ and $\P\left(\frac{|S_n|}{n^\alpha}>\epsilon\right)\to 0$, i.e. $\frac{S_n}{n^\alpha}\xrightarrow[n\to \infty]{ P} 0$. Hence, by Theorem~3.1 of Li and Rosalsky (2006), we have
    $\ds{\limsup_{n\to\infty}}\frac{T_n}{n^{\alpha}}\leqslant\epsilon\quad \text{a.s.  for all } \epsilon >0$,
 and then, by letting $\epsilon \downarrow 0$, we get the desired result, and this completes the proof.
\end{proof}
Note that, by part (3) of Theorem~\ref{thm:li3.2}, we generalizes Theorem 3.2 of Li and Rosalsky~(2006). In addition, we present a very short proof of the equivalence between part (1) and part (2) as compared to the proof given in the quoted paper.
\begin{thm}\label{thm:li3.3}
  Suppose that $n/p_n$ is bounded away from $0$ and $\infty$. \\If
    $\E (U_{1,1})\E (V_{1,1}) = 0, \quad{}\E (U_{1,1}^2)\E (V_{1,1}^2) = 1$
  and\\
  \begin{equation}
    \E\left( \frac{|U_{1,1}V_{1,2}|^6}{\left(\log(e+|U_{1,1}V_{1,2}|)\right)^3} \right) < \infty \,\mbox{ or }\,\sum_{n=1}^\infty\P\left(\max_{1\leq i\neq j\leq n}|U_{1,i}V_{1,j}|\geq\sqrt{n\log n}\right) <\infty,\label{lemaDeli}
  \end{equation}
then
\begin{equation*}
  \limsup_{n\to\infty}\frac{T_n}{\sqrt{n\log n}}\leq 2\quad \text{a.s.}
\end{equation*}
Conversely, if $\ds{\limsup_{n\to\infty}}\frac{T_n}{\sqrt{n\log n}}<\infty$ a.s., then \eqref{lemaDeli} hold, $\E (U_{1,1})\E (V_{1,1}) = 0$, and \\$\E\left( |U_{1,1}|^\beta\right)\E\left(|V_{1,2}|^\beta \right) < \infty$ for all $0\leqslant\beta<6$.
\end{thm}
\begin{proof} The proof of the second part of the theorem follows from Corollary~\ref{correm23} and by following the same steps as in proof of the only if part of Theorem~3.3 of Li and Rosalsky~(2006). To prove the first part of the theorem, we note first that
the equivalence between the conditions in~\eqref{lemaDeli} follows directly from
 Corollary~\ref{cormainl:2} by taking $m=2$. Further, if
  \begin{equation*}
    \E\left( \frac{|U_{1,1}V_{1,2}|^6}{\left(\log(e+|U_{1,1}V_{1,2}|)\right)^3} \right) < \infty,
  \end{equation*}
 by Theorem~3 of Lai~(1974) (or see Theorem~\ref{thm:lai} in the Appendix~\ref{sec:appendB}), we have
  \begin{equation*}
    \sum_{n=2}^\infty n\P\left(\frac{|S_n|}{\sqrt{n\log n}}>\lambda\right)<\infty\quad\text{for all }\lambda>2,
  \end{equation*}
  and then, $\ds{\frac{S_n}{\sqrt{n\log n}}}\xrightarrow[n\to \infty]{ P} 0$ and $\ds{\sum_{n=2}^\infty} \frac{p_{n}^{2}}{n}\P\left(\frac{|S_n|}{\sqrt{n\log n}}>\lambda\right)<\infty\quad\text{for all }\lambda>2$. Therefore, by Theorem 3.1 of Li and Rosalsky (2006), we get $ \ds{\limsup_{n\to\infty}}\frac{T_n}{n^{\alpha}}\leqslant\lambda\quad \text{a.s.  for all } \lambda>2$,
 and then, letting $\lambda \downarrow 2$, we get the desired result, and this completes the proof.
\end{proof}

  Note that by taking $\beta=2$ in the second part of Theorem~\ref{thm:li3.3}, we have the statement in Li and Rosalsky~(2006). Further, in the first part of Theorem~\ref{thm:li3.3}, the condition~\eqref{lemaDeli} relaxes the condition~(3.15) of Li and Rosalsky~(2006). Moreover, in addition to state in its full generality the result of Theorem~3.3 of Li and Rosalsky~(2006), we simplify substantially the proof. By using Corollary~\ref{cormain:1}, we establish the following result which generalizes Theorem~2.1 in Li and Rosalsky~(2006).
\begin{thm} \label{thm:1}
   Suppose that $n/p_n$ is bounded away from $0$ and $\infty$. Let $1/2<\alpha \leq 1$. Then, the following statements are equivalent.
  \begin{enumerate}
    \item[(1).] 
  $\ds{\sum_{n=1}^\infty} n^2\P(|X_{1,1}X_{1,2}|\geq n^\alpha) < \infty \quad \mbox{\em and }\quad \E(X_{1,1}) = 0$.
    \item[(2).]   
   $\ds{\sum_{n=1}^\infty}\P(\max_{1\leq i<j\leq n}|X_{1,i}X_{1,j}|\geq n^\alpha) < \infty \quad \mbox{ \em and }\quad \E(X_{1,1}) = 0$.
    \item[(3).]  
   $\ds{\lim_{n\to\infty}}\frac{W_n}{n^\alpha} = 0. \quad\text{a.s.}$
    \item[(4).] 
  $\E\Big(|X_{1,1}|^{3/\alpha}\Big) < \infty \quad \mbox{\em and } \quad \E(X_{1,1}) = 0$.
  \end{enumerate}
\end{thm}
\begin{proof}[Proof]
The equivalence between (1) and (2) follows directly form Corollary~\ref{cormain:1} by taking $u_{n}=n^{\alpha}$. Further, the equivalence between the statements in $(2)$ and $(3)$ is established in Li and Rosalsky~(2006). Finally, the equivalence between the statements (2) and (4) follows from Corollary~\ref{cormainlog} by taking $m=2$, this completes the proof. 
  \end{proof}
Further, by using Lemma~\ref{lem:1} and Corollary~\ref{cormainl:2}, we establish in its full generality Theorem~2.3 of Li and Rosalsky~(2006).
\begin{thm}
  Suppose that $n/p_n$ is bounded away from $0$ and $\infty$. Then, the following statements are equivalent.
  \begin{enumerate}
    \item[(1).]     
     $\E(X_{1,1}) = 0, \quad \E(X_{1,1}^2) = 1 \text{\em  and  }  \ds{\sum_{n=1}^\infty} n^2\P\bigg(|X_{1,1}X_{1,2}|\geq \sqrt{n\log n} \bigg)<\infty.$
    \item[(2).]
    $\E(X_{1,1}) = 0, \quad \E(X_{1,1}^2) = 1  \text{ \em and  }  \ds{\sum_{n=1}^\infty}\P\bigg( \max_{1\leq i<j\leq n}|X_{1,i}X_{1,j}|\geq \sqrt{n\log n} \bigg)<\infty.$
    \item[(3).]
    $\ds{\lim_{n\to\infty}}\frac{W_n}{\sqrt{n\log n}} = 2 \text{  a.s.}$
    \item[(4).]
    $\E(X_{1,1}) = 0, \quad \E(X_{1,1}^2) = 1  \text{\em and  }  \E\Big( \ds{\frac{(X_{1,1}X_{1,2})^6}{\log^3(e + |X_{1,1}X_{1,2}|)}} \Big) < \infty$.
  \end{enumerate}
\end{thm}
\begin{proof}[Proof]
 The Equivalence between $(1)$ and $(2)$ follows directly from Lemma~\ref{lem:1} by taking $m=2$ and $u_{n}=\sqrt{n\ln(n)}$. Further, the equivalence between $(2)$ and $(3)$ is established in Li and Rosalsky~(2006), and  the equivalence between $(2)$ and $(4)$ follows directly from Corollary \ref{cormainl:2} by taking $m=2$, this completes the proof.
\end{proof}
\appendix
\section{Proofs of some preliminary results}\label{sec:appendA}
  \begin{proof}[Proof of Theorem~\ref{thm:2}]
  Let $\Lambda_n = \{\beta_n\leq|X|<\beta_{n+1}\}$, $n=0,1,\dots$. We have
  \begin{equation*}
    \E(|X|) =\ds{ \int_{\ds{\bigcup_{n=0}^\infty} \Lambda_n}}|X|\,d\P = \sum_{n=0}^\infty \int_{\Lambda_n}|X|\,d\P,
  \end{equation*}
  and then,
  \begin{equation}
    \sum_{n=0}^\infty\beta_n\P(\Lambda_n) \leq \E(|X|) \leq\sum_{n=0}^\infty \beta_{n+1}\P(\Lambda_n) = \sum_{n=0}^\infty (\beta_{n+1} - \beta_{n})\P(\Lambda_n) + \sum_{n=0}^\infty\beta_n\P(\Lambda_n).\label{eqclein}
  \end{equation}
  Further, there exists $B>0$ such that $(\beta_{n+1} - \beta_{n})\leqslant B\beta_{n}$, $n=1,2,\dots$. Then,
  \begin{eqnarray}
    \sum_{n=0}^\infty (\beta_{n+1} - \beta_{n})\P(\Lambda_n) + \sum_{n=0}^\infty\beta_n\P(\Lambda_n)\leqslant \beta_{1}\P(\Lambda_{0})+B \sum_{n=1}^\infty\beta_n\P(\Lambda_n)+ \sum_{n=1}^\infty\beta_n\P(\Lambda_n)\nonumber\\=\beta_{1}\P(\Lambda_{0})+(B+1) \sum_{n=1}^\infty\beta_n\P(\Lambda_n).\nonumber 
  \end{eqnarray}
  Then, from \eqref{eqclein}, we have
  \begin{equation}
    \sum_{n=1}^\infty\beta_n\P(\Lambda_n) \leq \E(|X|) \leq \beta_{1}+(B+1) \sum_{n=1}^\infty\beta_n\P(\Lambda_n).\label{eqcle}
  \end{equation}
  Observe that
  \begin{align}
    \sum_{n=1}^N\beta_n\P(\Lambda_n) 
    & = \sum_{n=1}^N (\beta_{n} -\beta_{n-1})\P(|X|\geq \beta_n) - \beta_N\P(|X|\geq\beta_{N+1}), \quad N=1, 2,\dots, \label{eqcle2}
  \end{align}
  then,
  \begin{align}
    \sum_{n=1}^N\beta_n\P(\Lambda_n) \leqslant \sum_{n=1}^N (\beta_{n} -\beta_{n-1})\P(|X|\geq \beta_n), \quad N=1, 2,\dots, \label{inqcle2}
  \end{align}
  this gives
  \begin{align}
    \sum_{n=1}^\infty\beta_n\P(\Lambda_n) \leqslant \sum_{n=1}^\infty (\beta_{n} -\beta_{n-1})\P(|X|\geq \beta_n)\leqslant c\sum_{n=1}^\infty \alpha_{n}\P(|X|\geq \beta_n), \label{inqcle3}
  \end{align}
  and then, combining \eqref{eqcle} and \eqref{inqcle3}, we have
  \begin{equation}
    \sum_{n=1}^\infty\beta_n\P(\Lambda_n) \leq \E(|X|) \leq \beta_{1}+(B+1) \sum_{n=1}^\infty\beta_n\P(\Lambda_n)\leqslant\beta_{1}+(B+1)c\sum_{n=1}^\infty \alpha_{n}\P(|X|\geq \beta_n).\label{inqcle4}
  \end{equation}
  First, suppose that $\E(|X|)=\infty$, by \eqref{inqcle4}, we have $\ds{\sum_{n=1}^\infty }\alpha_{n}\P(|X|\geq \beta_n)=\infty$ and then, \\$c^{-1}\ds{\sum_{n=1}^\infty } \alpha_{n}\P(|X|\geq \beta_n)=\E(|X|)=(B+1)c\ds{\sum_{n=1}^\infty} \alpha_{n}\P(|X|\geq \beta_n)=\infty$, this proves the statement.

  Second, suppose that $\E(|X|)< \infty$.
  By Lebesgue dominated convergence Theorem, we have $\ds{\lim_{N\rightarrow \infty}}\E(|X|\mathbb{I}_{\{ |X|\geq \beta_{N+1} \}})= 0$ and then,
  \begin{equation*}
    \beta_N\P(|X|\geq\beta_{N+1}) \leq \beta_{N+1}\P(|X|\geq\beta_{N+1}) \leq \E(|X|\mathbb{I}_{\{ |X|\geq \beta_{N+1} \}})\xrightarrow[N\to \infty]{  } 0.
  \end{equation*}
  Thus, from \eqref{eqcle2}, we have
  \begin{equation}
    \sum_{n=1}^\infty (\beta_{n} -\beta_{n-1})\P(|X|\geq \beta_n)=\sum_{n=1}^\infty\beta_n\P(\Lambda_n), \label{eqclefin}
  \end{equation}
  then, since $c^{-1}\alpha_{n}\leqslant \beta_n - \beta_{n-1} \leqslant c \alpha_n$, by combining \eqref{inqcle4} and \eqref{eqclefin},
  we have
  \begin{equation*}
    c^{-1}\sum_{n=1}^\infty \alpha_n\P(|X|\geq \beta_n) \leqslant \sum_{n=1}^\infty\beta_n\P(\Lambda_n)\leqslant \E(|X|)\leqslant \beta_{1}+(B+1)c\sum_{n=1}^\infty \alpha_n\P(|X|\geq \beta_n,
  \end{equation*}
this completes the proof.
%
%
\end{proof}

\begin{proof}[Proof of Corollary~\ref{coro:1}]
  The condition that $\displaystyle{\sum_{n=1}^\infty n^\alpha}\P\Big( |X|>n^\beta \Big) < \infty$ is equivalent with\\ $\displaystyle{\sum_{n=1}^\infty} n^\alpha\P\Big(|X|^{\frac{\alpha+1}{\beta}}>n^{\alpha+1}\Big) < \infty$. Now let $\alpha_n := n^\alpha$ and $\beta_n := n^{\alpha+1}$. We have \\$\beta_{n+1} - \beta_n = (n+1)^{\alpha+1} -n^{\alpha+1}=n^{\alpha+1}\Big( (\frac{n+1}{n})^{\alpha+1}-1 \Big)\leq(2^{\alpha+1}-1)\cdot n^{\alpha+1}=\mathcal{O}(n^{\alpha+1})$, and
  \begin{align*}
    \beta_{n} - \beta_{n-1}& = n^{\alpha+1} - (n-1)^{\alpha+1} = n^\alpha\Big[n-(\frac{n-1}{n})^\alpha\cdot(n-1)\Big]
    \geq n^\alpha = \alpha_n
  \end{align*}
  Also, since $n - (\frac{n-1}{n})^\alpha(n-1)\to 1$, by the fact every convergent sequence is bounded, there exists a constant $C$ such that for all $n\geq1$, $n - (\frac{n-1}{n})^\alpha(n-1)\leq C$. Let $c = \max\{2,C\}$. We verified that $c^{-1}\alpha_n\leq \beta_n - \beta_{n-1} \leq c\alpha_n \quad \forall n, \text{  for some } c\geq 1$. Therefore, the rest of the proof follows directly from Theorem~\ref{thm:2}.
\end{proof}

\begin{proof}[Proof of Corollary~\ref{coro:2}]
  The proof follows form Corollary \ref{coro:1}. Since $\alpha_n/n^\alpha$ and $\beta_n/n^\beta$ are bounded away from $0$ and $\infty$. Then exists $a>1$ s.t. $a^{-1}<\frac{\alpha_n }{n^\alpha}<a$ and $a^{-1}<\frac{\beta_n }{n^\beta}<a$. Thus $\displaystyle{\sum_{n=1}^\infty \alpha_n}\P\Big(|X|>\beta_n\Big)<\infty$ implies $\displaystyle{\sum_{n=1}^\infty a^{-1}n^\alpha\P\Big( |X|>an^\beta \Big)}<\infty$, and this is equivalent to $\displaystyle{\sum_{n=1}^\infty n^\alpha\P\Big( \frac{|X|}{a}>n^\beta \Big)}<\infty.$ Now by Corollary \ref{coro:1}, $\E\Big(\left(\frac{|X|}{a}\right)^{\frac{\alpha+1}{\beta}}\Big)<\infty$, i.e. $\E\Big(|X|^{\frac{\alpha+1}{\beta}}\Big)<\infty$. Conversely, if $\E(|X|^{\frac{\alpha+1}{\beta}})<\infty$, we have $\E\Big((a\cdot|X|)^\frac{\alpha+1}{\beta}\Big) < \infty$. Then, by Corollary \ref{coro:1}, $\displaystyle{\sum_{n=1}^\infty n^\alpha\P\Big(|X|>\frac{n^\beta}{a}\Big)} < \infty$, which implies $\displaystyle{\sum_{n=1}^\infty \alpha_n\P\Big(|X|>\beta_n\Big)}$, this complete the proof.
\end{proof}
\begin{proof}[Proof of Lemma~\ref{lem:1}]
The sufficient condition follows directly from the sub-additivity. To prove the necessary condition, let
$\mathbb{A}=\{n:\P\Big(\ds{\max_{1\leq i_{1}<i_{2}<\dots<i_{m}\leq n}}\prod_{h=1}^{m}|X_{1,i_{h}}|\geq u_{n} \Big)=0\}$. We have
\begin{eqnarray}
\ds{\sum_{n=1}^{\infty}}\P\Big(\max_{1\leq i_{1}<i_{2}<\dots<i_{m}\leq n}\prod_{h=1}^{m}|X_{1,i_{h}}|\geq u_{n} \Big)=\ds{\sum_{n\in \mathbb{N}\backslash \mathbb{A}}}\P\Big(\max_{1\leq i_{1}<i_{2}<\dots<i_{m}\leq n}\prod_{h=1}^{m}|X_{1,i_{h}}|\geq u_{n} \Big)\label{nonnullmax}
\end{eqnarray}
and note that if $n\in \mathbb{A}$, $\P\Big(\ds{\prod_{h=1}^{m}}|X_{1,h}|\geq u_n\Big)=0$ and then $n^m\P\Big(\ds{\prod_{h=1}^{m}}|X_{1,h}|\geq u_n\Big)=0$, $\forall n\in \mathbb{A}$. Then,
\begin{eqnarray}
\ds{\sum_{n=1}^{\infty}}n^m\P\Big(\prod_{h=1}^{m}|X_{1,h}|\geq u_n\Big)=\ds{\sum_{n\in \mathbb{N}\backslash \mathbb{A}}}n^m\P\Big(\prod_{h=1}^{m}|X_{1,h}|\geq u_n\Big). \label{nonnullad}
\end{eqnarray}
Hence, from \eqref{nonnullmax}~and~\eqref{nonnullad}, it suffices  to prove that if \\$\ds{\sum_{n\in \mathbb{N}\backslash \mathbb{A}}}\P\Big(\ds{\max_{1\leq i_{1}<i_{2}<\dots<i_{m}\leq n}}\prod_{h=1}^{m}|X_{1,i_{h}}|\geq u_{n} \Big)<\infty$ then $\ds{\sum_{n\in \mathbb{N}\backslash \mathbb{A}}}n^m\P\Big(\ds{\prod_{h=1}^{m}}|X_{1,h}|\geq u_n\Big)<\infty$. Hence, in the sequel, we suppose without loss of generality that $\P\Big(\ds{\prod_{h=1}^{m}}|X_{1,h}|\geq u_n\Big)>0$ for all $n=1,2,\dots$
Thus, if $\ds{\sum_{n=1}^{\infty}}\P\Big(\max_{1\leq i_{1}<i_{2}<\dots<i_{m}\leq n}\prod_{h=1}^{m}|X_{1,i_{h}}|\geq u_{n} \Big)<\infty$, we have
\begin{eqnarray}
\ds{\lim_{n\rightarrow \infty}}\P\Big(\max_{1\leq i_{1}<i_{2}<\dots<i_{m}\leq n}\prod_{h=1}^{m}|X_{1,i_{h}}|\geq u_{n} \Big)=0. \label{eqimport}
\end{eqnarray}
 Then, it suffices to prove that
  \begin{equation*}
    m!\P\Big(\max_{1\leq i_{1}<i_{2}<\dots<i_{m}\leq n}\prod_{h=1}^{m}|X_{1,i_{h}}|\geq u_{n} \Big) \sim n^m\P\Big(\prod_{h=1}^{m}|X_{1,h}|\geq u_n\Big),\quad \text{as } n\to\infty.
  \end{equation*}
  For convenience $X_{1,i}$ is written as $X_i$.

First, observe that $\log x\leq x-1, \forall x> 0$. Then,
\begin{equation*}
  \log\P\Big(\prod_{h=1}^{m}|X_{h}|\leq u_n \Big) \leq \P\Big( \prod_{h=1}^{m}|X_{h}|\leq u_n \Big) -1
  = -\P\big( \prod_{h=1}^{m}|X_{h}|>u_n \big),
\end{equation*}
and then,
\begin{equation*}
  \frac{n!}{(n-m)!m!}\log \P\left(\prod_{h=1}^{m}|X_{h}|\leq u_n\right) \leq -\frac{n!}{(n-m)!m!}\P(\prod_{h=1}^{m}|X_{h}|>u_n).
\end{equation*}
Taking the exponential both side, we have
\begin{equation*}
  \Big[\P\big(\prod_{h=1}^{m}|X_{h}|\leq u_n\big)\Big]^{\ds{\frac{n!}{(n-m)!m!}}} \leq \exp\Big[ -\frac{n!}{(n-m)!m!}\P(\prod_{h=1}^{m}|X_{h}|>u_n) \Big].
\end{equation*}
Thus,
\begin{equation*}
  \P\Big( \max_{1\leq i_{1}<i_{2}<\dots<i_{m}\leq n}\prod_{h=1}^{m}|X_{1,i_{h}}|\leq u_n \Big) \leq \exp\Big[ -\frac{n!}{(n-m)!m!}\P(\prod_{h=1}^{m}|X_{h}|>u_n) \Big],
\end{equation*}
and then
\begin{equation}
  \P\Big( \max_{1\leq i_{1}<i_{2}<\dots<i_{m}\leq n}\prod_{h=1}^{m}|X_{1,i_{h}}| > u_n \Big) \geq 1 - \exp\left[ -\frac{n!}{(n-m)!m!}\P\left(\prod_{h=1}^{m}|X_{h}|>u_n\right) \right]>0.\label{inequality}
\end{equation}
Then, by combining \eqref{eqimport} and \eqref{inequality}, we get
\begin{equation}
  \lim_{n\rightarrow \infty}\P\Big( \max_{1\leq i_{1}<i_{2}<\dots<i_{m}\leq n}\prod_{h=1}^{m}|X_{1,i_{h}}| > u_n \Big) = \lim_{n\rightarrow \infty}\frac{n!}{(n-m)!m!}\P(\prod_{h=1}^{m}|X_{h}|>u_n) =0.\label{equivelent}
\end{equation}
By combining ~\eqref{inequality} and ~\eqref{equivelent} along with the fact that  $1 - e^{-x} \sim x$ as $x\to 0$, we get
\begin{equation}
  \lim_{n\rightarrow \infty}\ds{\frac{\P\Big( \ds{\max_{1\leq i_{1}<i_{2}<\dots<i_{m}\leq n}}\prod_{h=1}^{m}|X_{1,i_{h}}| > u_n \Big) }{\ds{\frac{n!}{(n-m)!m!}\P(\prod_{h=1}^{m}|X_{h}|>u_n) }}}\geq 1.\label{preliequi}
\end{equation}
Further, by sub-additivity, we have
\begin{equation*}
  \frac{n!}{(n-m)!m!}\P(\prod_{h=1}^{m}|X_{h}|>u_n) \geqslant \P\Big( \max_{1\leq i_{1}<i_{2}<\dots<i_{m}\leq n}\prod_{h=1}^{m}|X_{1,i_{h}}| > u_n \Big),
\end{equation*}
and then,
\begin{equation}
  \lim_{n\rightarrow \infty}\frac{\P\Big( \ds{\max_{1\leq i_{1}<i_{2}<\dots<i_{m}\leq n}}\ds{\prod_{h=1}^{m}}|X_{1,i_{h}}| > u_n \Big) }{\ds{\frac{n!}{(n-m)!m!}}\P\left(\ds{\prod_{h=1}^{m}}|X_{h}|>u_n\right) }\leqslant1.\label{prelsubadd}
\end{equation}
Hence, by combining ~\eqref{preliequi} and \eqref{prelsubadd}, we have
\begin{equation*}
  \frac{n!}{(n-m)!m!}\P\left(\prod_{h=1}^{m}|X_{h}|>u_n\right) \sim \P\Big( \max_{1\leq i_{1}<i_{2}<\dots<i_{m}\leq n}\prod_{h=1}^{m}|X_{1,i_{h}}| > u_n \Big).
\end{equation*}
Therefore,
\begin{equation*}
  \frac{n^m}{m!}\P\left(\prod_{h=1}^{m}|X_{h}|>u_n\right) \sim \P\Big( \max_{1\leq i_{1}<i_{2}<\dots<i_{m}\leq n}\prod_{h=1}^{m}|X_{1,i_{h}}| > u_n \Big),
\end{equation*}
this completes the proof.
  \end{proof}
\begin{proof}[Proof of Corollary \ref{cormainl:2}]
 By Lemma \ref{lem:1}, $\ds{\sum_{n=1}^\infty}\P\left(\max_{1\leq i_{1}<i_{2}<\dots<i_{m}\leq n}\prod_{h=1}^{m}|X_{1,i_{h}}|\geq \sqrt{n \ln(n)}\right) <~\infty$ if and only if
$\ds{\sum_{n=3}^\infty} n^m\P\left(\prod_{h=1}^{m}|X_{1,h}|\geq \sqrt{n\ln(n)}\right)< \infty$, and this is equivalent to \\$\ds{\sum_{n=3}^\infty} n^m\P\left(\prod_{h=1}^{m}|X_{1,h}|^{2}\geq n\ln(n)\right)< \infty$. Then, by taking $f(x)=\frac{x}{\ln\left(e+\sqrt{x}\right)}$, $x\geqslant3$, one can verify that $f(n\ln(n))\big/n\geqslant 1/2$ for all $n\geqslant3$ and $\ds{\lim_{n\rightarrow\infty}}f(n\ln(n))\big/n=2$ and this implies that $f(n\ln(n))\big/n\geqslant 1/2$ is bounded away from 0 and from infinity. Therefore, the proof follows from Lemma~\ref{lem:11} by taking $\beta=1$.
\end{proof}

\section{Theorems of Baum and Katz~(1965) and Lai~(1974) used}\label{sec:appendB}
For the convenience of the reader, we recall in this Appendix two results from the celebrated theorems of Baum and Katz~(1965), and Lai~(1974) which are used in this paper. These results can also been found in Li and Rosalsky~(2006).
\begin{thm}[Baum and Katz~(1965)]\label{thm:BaumKatz}
 Let $\{S_{n},n\geqslant 1\}$ be a sequence of partial sums as defined in Section~\ref{sec:mainres}. Let $\beta>0$ and $\alpha>1/2$, and suppose that $\E(Y_{1})=0$ if $\alpha\leqslant 1$. Then, the following are equivalent:
\begin{eqnarray*}
\sum_{n=1}^{\infty} n^{2\beta-1}\P\left(\frac{|S_{n}|}{n^{\alpha}}>\epsilon\right)&<&\infty \quad{ } \mbox{ for all $\epsilon>0$},\\
\sum_{n=1}^{\infty} n^{2\beta-1}\P\left(\ds{\sup_{m\geqslant n}}\frac{|S_{m}|}{m^{\alpha}}>\epsilon\right)&<&\infty \quad{ } \mbox{ for all $\epsilon>0$},\\
\E\left(|Y_{1}|^{(2\beta+1)/\alpha}\right)&<&\infty.
\end{eqnarray*}
\end{thm}
\begin{thm} [Lai~(1974)] \label{thm:lai}
 Let $\beta>0$, let $S_{n}$ be as in Theorem~\ref{thm:BaumKatz} and suppose that
\begin{eqnarray*}
\E(Y_{1})=0, \E(Y_{1}^{2})=1, \mbox{ and } \E\left(\frac{|Y_{1}|^{4\beta+2}}{\left(\ln\left(e+|Y_{1}|\right)\right)^{2\beta+1}}\right)<\infty.
\end{eqnarray*}
Then,
\begin{eqnarray*}
\sum_{n=1}^{\infty} n^{2\beta-1}\P\left(\frac{|S_{n}|}{\sqrt{n\ln(n)}}>\lambda\right)&<&\infty \quad{ } \mbox{ for all $\lambda>2\sqrt{\beta}$}.
\end{eqnarray*}
\end{thm}

\end{document}